\documentclass[11pt]{amsart}
\usepackage{amsmath,amssymb,amsfonts,amsthm,amscd}

\usepackage[dvipdfm]{hyperref}

\newtheorem{prop}{Proposition}
\newtheorem{thm}[prop]{Theorem}

\newtheorem{defi}[prop]{Definition}

\renewcommand{\phi}{\varphi}
\renewcommand{\epsilon}{\varepsilon}

\begin{document}
\title[A priori estimates of the degenerate Monge-Ampere equation]{A priori estimates of the degenerate Monge-Ampere equation on Kahler manifolds of nonnegative bisectional curvature}
\author{Sebastien Picard}
\address{Department of Mathematics\\
         McGill University\\
         Montreal, Quebec. H3A 2K6, Canada.}
\email{sebastien.picard@math.mcgill.ca}
\begin{abstract}
The regularity theory of the degenerate complex Monge-Amp\`{e}re equation is studied. The equation is considered on a closed compact K\"{a}hler manifold $(M,g)$ with nonnegative orthogonal bisectional curvature of dimension $m$. Given a solution $\phi$ of the degenerate complex Monge-Amp\`{e}re equation $\det(g_{i \bar{j}} + \phi_{i \bar{j}}) = f \det(g_{i \bar{j}})$, it is shown that the Laplacian of $\phi$ can be controlled by a constant depending on $(M,g)$, $\sup f$, and $\inf_M \Delta f^{1/(m-1)}$.
\end{abstract}

\date{}
\maketitle

\section{Introduction}

\indent We will be looking at the regularity theory of the degenerate complex Monge-Amp\`{e}re equation. Let us consider the equation on a compact K\"{a}hler manifold $(M,g)$ without boundary of dimension $m$. The problem of solving the complex Monge-Amp\`{e}re equation on $M$ was first motivated by the Calabi conjecture. The conjecture was reduced to solving a non-degenerate Monge-Amp\`{e}re equation, and the question was solved by Yau in ~\cite{yau}. This result had significant geometric implications, in particular, leading to the theory of Calabi-Yau manifolds, which now plays a central role in string theory and complex geometry. One of the major steps in the proof of the Calabi conjecture was establishing an {\it a priori} estimate on the Laplacian of the solution; this was done independently by Aubin ~\cite{aubin} and Yau ~\cite{yau}. 
\newline
\indent Although the Calabi conjecture deals with a non-degenerate Monge-Amp\`{e}re equation, Yau's paper ~\cite{yau} also treated the degenerate Monge-Amp\`{e}re equation, with an application to holomorphic sections of line bundles over $M$. More recently, the results of Yau were generalized by Kolodziej ~\cite{kol}. The degenerate Monge-Amp\`{e}re equation in complex geometry has become an active area of research, with connections to the Minimal Model Program ~\cite{EGZ, BEGZ}, or geodesics joining two K\"{a}hler potentials in the space of K\"{a}hler metrics ~\cite{donaldson, chen, phong_sturm}. For a survey of some of these topics, see ~\cite{phong_song_sturm}.
\newline
\indent In this article, we shall consider the following complex Monge-Amp\`{e}re equation:
\begin{equation} \label{mongeampere} \det (\phi_{i \bar{j}} + g_{i \bar{j}}) = f \det g_{i \bar{j}} ,\end{equation}
\noindent where $f: M \rightarrow \mathbb{R}$ and $f \geq 0$.
\newline
\indent The objective is the following: given a solution $\phi$ to \eqref{mongeampere} such that $(\phi_{i \bar{j}} + g_{i \bar{j}})$ is positive semi-definite, we seek an estimate $|\Delta \phi| \leq C$ depending only on $(M,g)$, $\sup f$ and a constant $A$ such that
\begin{equation} \label{la_condition} \inf_M \Delta f^{\frac{1}{m-1}} \geq -A.\end{equation}
\indent This problem is motivated by a similar result obtained for the real Monge-Amp\`{e}re equation by P. Guan in ~\cite{guan}. Previously, a result on K\"{a}hler manifolds was obtained by Blocki in ~\cite{blocki2} while assuming that $f^{\frac{1}{m-1}} \in C^{1,1}$. Later, Blocki improved his result in ~\cite{blocki} to requiring the assumption that $f^{\frac{1}{m-1}} \Delta (\log f)$ is bounded below, which is equivalent to 
$$ f \Delta f - |\nabla f|^2 \geq -A f^{2-\frac{1}{m-1}}.$$
\indent In comparison, our desired condition \eqref{la_condition} is equivalent to
$$ f \Delta f - \frac{m-2}{m-1}|\nabla f|^2 \geq - Af^{2-\frac{1}{m-1}}.$$
\indent We solve the problem in the case when $M$ has nonnegative orthogonal bisectional curvature. Similar results were obtained by Hou in ~\cite{hou} by obtaining a Laplacian estimate for complex Hessian equations depending on $\inf_M \Delta f^{\frac{1}{m}}$ in the case of nonnegative orthogonal bisectional curvature. The main difficulty in improving the exponent from $1/m$ to $1/(m-1)$ is that we can no longer discard terms by using the concavity of $(\det B)^{1/m}$, and the resulting third-order terms must be handled carefully. 
\newline
\indent We use the following definition.
\begin{defi} A K\"ahler manifold $(M,g)$ is said to have nonnegative orthogonal bisectional curvature if at each point $p \in M$, for any orthonormal basis $\{e_1, \dots, e_m \}$ of the tangent space at $p$, we have  $R(e_i,\overline{e_i},e_j,\overline{e_j}) = R_{i \bar{i} j \bar{j}} \geq 0$.
\end{defi}
\
\indent Manifolds satisfying this curvature condition are well-understood; indeed, compact K\"{a}hler manifolds of non-negative orthogonal bisectional curvature were classified by Gu and Zhang in ~\cite{gu_zhang}. Our main theorem is the following.
\begin{thm} \label{thm-pos-bisec} Let $(M,g)$ be a compact K\"ahler manifold with nonnegative orthogonal bisectional curvature and empty boundary. Let $f > 0$ be a function on $M$ such that $\inf_M \Delta f^{\frac{1}{m-1}} \geq -A$ for some constant $A$. For all solutions $\phi \in C^4(M)$ of 
$$ \det ( g_{i \bar{j}} + \phi_{i \bar{j}} ) = f \det g_{i \bar{j}},$$
\noindent such that $(\phi_{i \bar{j}} + g_{i \bar{j}})$ is positive definite, we have
$$ (\sup_M \phi - \inf_M \phi) + || \nabla \phi||_{\infty} + || \Delta \phi ||_{\infty} \leq C,$$
\noindent where $C$ depends on $(M,g)$, $A$, and $\sup f$. 
\end{thm}
\
\indent Although this theorem assumes $f>0$, such an estimate is applicable to the degenerate case when $f \geq 0$ via a limiting process. This shall be illustrated in Section 3, which contains a proof of the following application.
\newline
\indent Using the {\it a priori estimates}, we solve the Dirichlet problem for the degenerate complex Monge-Amp\`{e}re equation on a domain $\Omega$ in $\mathbb{C}^m$. This type of question was investigated by Bedford and Taylor in ~\cite{bedford_taylor}, and Caffarelli, Kohn, Niremberg and Spruck in ~\cite{ckns}. The problem is of the following form:
\begin{equation} \label{thm3} \det u_{i \bar{j}}(z) = f(z) \ \ {\rm in } \ \Omega,\end{equation}
$$ u = 0 \ \ {\rm on} \ \partial \Omega.$$
\indent Before stating our result, we first establish some terminology. We say a real-valued function $u$ is pluri-subharmonic if $(u_{i \bar{j}})$ is positive semi-definite. We say a real-valued function $u$ is strictly pluri-subharmonic if $(u_{i \bar{j}})$ is positive definite. Following the terminology of ~\cite{blocki2}, if $|\Delta u|$ is bounded, we say $u$ is almost $C^{1,1}$. A domain $\Omega \subset \mathbb{C}^m$ with smooth boundary $\partial \Omega$ is called strongly pseudo-convex if there exists a smooth real-valued function $r$ defined on a neighbourhood of $\overline{\Omega}$ such that $r<0$ in $\Omega$, $r=0$ on $\partial \Omega$, $r>0$ outside of $\overline{\Omega}$, $dr \neq 0$, and $(r_{i \bar{j}}(z))$ is positive-definite at each point in its domain. With these definitions in place, we can now state the following result.
\begin{thm} \label{dirichlet}
Let $\Omega$ be a bounded strongly pseudo-convex domain in $\mathbb{C}^m$. Let $f: \Omega \rightarrow \mathbb{R}$ be a function such that $f \geq 0$, $|\nabla f^{1/m}| \leq A_1$, and $\Delta f^{\frac{1}{m-1}} \geq -A_2$. Then there exists a unique pluri-subharmonic, almost $C^{1,1}$ solution $u$ such that 
$$ \det u_{i \bar{j}}(z) = f(z) \ \ {\rm in } \ \Omega,$$
$$ u = 0 \ \ {\rm on} \ \partial \Omega.$$
\indent Furthermore, $||u||_{C^1(\bar{\Omega})} + ||\Delta u||_\infty \leq C$, where $C$ depends only on $A_1$,$A_2$, $\sup(f)$ and $\Omega$.
\end{thm}

 \section*{Acknowledgements}
I would like to thank Prof. Pengfei Guan for suggesting this problem, and for his invaluable advice and guidance. I would also like to thank Xiangwen Zhang for pointing out several typos in the draft of this article.

\section{Preliminaries}

\indent In this section, we establish some notation and recall previous results relating to the proof of Theorem \ref{thm-pos-bisec}. First, we remind the reader that the constant $C$ denotes a positive quantity that is under control, and may change line by line. We will use the convention $\phi_{i \bar{j}} = \frac{\partial^2 \phi}{\partial z^i \partial \bar{z}^j}$; as opposed to ~\cite{yau}, subscripts do not indicate covariant derivatives. Also, for a function $h: M \rightarrow \mathbb{C}$, we will use any of the following notation interchangeably: $\frac{\partial h}{\partial z^i}$, $\partial_i h$, and $h_i$. We shall denote
\begin{equation} \label{g_primeprime} g'_{i \bar{j}} := g_{i \bar{j}} + \phi_{i \bar{j}}. \end{equation}
\indent It is well-known, as seen for example in the exposition of ~\cite{siu}, that for Theorem \ref{thm-pos-bisec} we have an {\it a priori} estimate 
$$ (\sup_M \phi - \inf_M \phi) \leq C.$$
\indent The objective of this article is to estimate $|\Delta \phi| \leq C$ directly from the uniform bound. Assuming such a bound on the Laplacian, we can obtain a bound on the gradient. Indeed, if we look at $\Delta \phi(z) := G(z)$, then $|G(z)| \leq C$. Then by the Schauder estimates
\begin{equation} \label{grad_est}
\sup_M | \nabla \phi| \leq C_0 (||G||_\infty + ||\phi||_2) \leq C.
\end{equation}
\indent Furthermore, we can easily obtain a lower bound on $\Delta \phi$. Since $g'_{i \bar{j}} = g_{i \bar{j}} + \phi_{i \bar{j}}$ is positive definite, we have $0 < Tr (g_{i \bar{j}} + \phi_{i \bar{j}})$. At any point $p \in M$, we may choose coordinates such that $g_{i \bar{j}} = \delta_{i \bar{j}}$. Thus
\begin{equation} \label{laplacian_phi} 0 < m + \Delta \phi, \end{equation}
\noindent and it only remains to bound $\Delta \phi$ from above.
\
\newline
\section{Second Order Estimate}

\indent As shown in the previous section, Theorem \ref{thm-pos-bisec} will following from the following estimate.

\begin{prop} \label{my_only_result} Let $(M,g)$ be a closed, compact K\"ahler manifold with nonnegative orthogonal bisectional curvature. Let $f>0$ be a positive function on $M$ such that $\inf_M \Delta f^{\frac{1}{m-1}} \geq -A$ for some constant $A$. For all $\phi \in C^4(M)$ satisfying \eqref{mongeampere} such that $(\phi_{i \bar{j}} + g_{i \bar{j}})$ is positive-definite, we have
$$ | \Delta \phi | \leq C,$$
\noindent where $C$ depends on $(M,g)$, $(\sup \phi - \inf \phi)$, $A$, and $\sup f$. 
\end{prop}

\begin{proof}

\indent We will estimate the maximum value of the following test function 
\begin{equation}
H = (m + \Delta \phi) e^{ - \alpha (\phi)},
\end{equation}
\noindent where $\alpha: [2, \lambda] \rightarrow \mathbb{R}$ is a function that will be specified later. In view of the $L^\infty$ estimate, we may shift $\phi$ by a constant and assume that $\phi(p) \in [2,\lambda]$ for all $p \in M$.  We start by computing the first two derivatives of $H$.
$$H_\gamma = (\Delta \phi)_\gamma e^{-\alpha(\phi)} - \alpha' (m+\Delta\phi)\phi_\gamma e^{-\alpha(\phi)}, $$
\begin{align}
H_{\gamma \bar{\gamma}} &= \left( (\Delta \phi)_{\gamma \bar{\gamma}} - \alpha' (m + \Delta \phi) \phi_{\gamma \bar{\gamma}} - \alpha'' (m+ \Delta \phi) \phi_\gamma \phi_{\bar{\gamma}} \right) e^{- \alpha(\phi)} \notag \\
&+ \left( - \alpha' \left( \ (\Delta \phi)_\gamma \phi_{\bar{\gamma}} + \overline{(\Delta \phi)_\gamma} \ \phi_{\gamma} \right) + (\alpha')^2 (m+ \Delta \phi) \phi_\gamma \phi_{\bar{\gamma}} \right) e^{- \alpha(\phi)} . \label{Hgammagamma}
\end{align}
\indent Let $p \in M$ be the point where $H$ achieves its maximum value. Since the manifold is K\"{a}hler, we may choose coordinates such that at $p$ we have $g_{i \bar{j}} = \delta_{ij}$, $\frac{\partial}{\partial z^k} g_{i \bar{j}} = 0$ and $\phi_{i \bar{j}} = \delta_{ij} \phi_{i \bar{j}}$. At $p$, the gradient of $H$ is equal to zero, and hence
\begin{equation} \label{gradientzero}
(\Delta \phi)_\gamma = \alpha' \phi_\gamma (m + \Delta \phi).
\end{equation}
\indent We recall the notation \eqref{g_primeprime}. Since $g'$ defines a K\"ahler metric on $M$, we denote $\Delta' = g'^{i \bar{j}} \partial_i \bar{\partial_j}$ to be the Laplacian of $(M, g')$. By the maximum principle, $\Delta' H (p) \leq 0$, hence if we use this fact while substituting the gradient equation \eqref{gradientzero} into \eqref{Hgammagamma}, we obtain
\begin{align}
\Delta' \Delta \phi &\leq \alpha' (m + \Delta \phi) \Delta' \phi + (\alpha')^2 (m + \Delta \phi) g'^{i \bar{i}} \phi_i \phi_{\bar{i}} \notag \\
&+ \alpha'' (m + \Delta \phi)g'^{i \bar{i}} \phi_i \phi_{\bar{i}} . \label{maxprince}
\end{align}
\indent Next, we raise both sides of \eqref{mongeampere} to the power of $1/(m-1)$ and take derivatives. This yields
\begin{align*}
&(m-1)(\det g_{i \bar{j}})^{\frac{1}{m-1}} \partial_\gamma f^{\frac{1}{m-1}} + f^{\frac{1}{m-1}} (\det g_{i \bar{j}})^{\frac{1}{m-1}} g^{i \bar{j}} \partial_\gamma g_{i \bar{j}}\\
&=  (\det g'_{i \bar{j}})^{\frac{1}{m-1}} g'^{i \bar{j}} ( \partial_\gamma g_{i \bar{j}} + \phi_{i \bar{j} \gamma}).
\end{align*}
\indent We then take another derivative of the previous equation, noting that our choice of coordinates will simplify the expression.
\begin{align*}
&(m-1)\partial_\gamma \bar{\partial}_\gamma f^{\frac{1}{m-1}} + f^{\frac{1}{m-1}} \delta^{i \bar{j}} \partial_\gamma \bar{\partial}_\gamma g_{i \bar{j}}\\
&= \frac{1}{m-1} f^{\frac{1}{m-1}} g'^{i \bar{i}} g'^{j \bar{j}} \phi_{i \bar{i} \gamma} \phi_{j \bar{j} \bar{\gamma}}\\
&+ f^{\frac{1}{m-1}}(g'^{i \bar{i}} (\partial_\gamma \bar{\partial}_\gamma g_{i \bar{i}} + \phi_{i \bar{i} \gamma \bar{\gamma}} ) + \phi_{i \bar{j} \gamma} \bar{\partial}_\gamma g'^{i \bar{j}} ).
\end{align*}
\indent Expanding out $\partial_\gamma g'^{i \bar{j}}$ and using the definition of the curvature tensor, we obtain
\begin{align}
&(m-1)f^{\frac{-1}{m-1}} \partial_\gamma \bar{\partial}_\gamma f^{\frac{1}{m-1}} - \delta^{i \bar{j}} R_{i \bar{j} \gamma \bar{\gamma}} \notag\\
&= \frac{1}{m-1} g'^{i \bar{i}} g'^{j \bar{j}} \phi_{i \bar{i} \gamma} \phi_{j \bar{j} \bar{\gamma}} - g'^{i \bar{i}} R_{i \bar{i} \gamma \bar{\gamma}} + g'^{i \bar{i}} \phi_{i \bar{i} \gamma \bar{\gamma}} - g'^{i \bar{i}} g'^{j \bar{j}} \phi_{i \bar{j} \gamma} \phi_{j \bar{i} \bar{\gamma}}. \label{curvterms}
\end{align}
\indent Also, at the point in consideration we have
\begin{align*}
\Delta' \Delta \phi &= g'^{k \bar{l}} \partial_k \bar{\partial}_l (g^{i \bar{j}} \phi_{i \bar{j}})\\
&= g'^{k \bar{l}}  \phi_{i \bar{j}} \partial_k \bar{\partial}_l g^{i \bar{j}} + g'^{k \bar{l}} g^{i \bar{j}} \phi_{i \bar{j} k \bar{l}}\\
&= -g'^{k \bar{l}} g^{i \bar{t}} g^{n \bar{j}}  \phi_{i \bar{j}} \partial_k \bar{\partial}_l g_{n \bar{t}} + g'^{k \bar{l}} g^{i \bar{j}} \phi_{i \bar{j} k \bar{l}}\\
&= \sum_{k, i} g'^{k \bar{k}} \phi_{i \bar{i}} R_{i \bar{i} k \bar{k}} + \sum_{k, i} g'^{i \bar{i}} \phi_{i \bar{i} k \bar{k}}.
\end{align*}
\indent After summing the $\gamma$ in \eqref{curvterms} and substituting the previous identity, one obtains the following at the point $p$:
\begin{align*}
(m-1) f^{\frac{-1}{m-1}} \Delta f^{\frac{1}{m-1}} &= \Delta' \Delta \phi + \sum_k \frac{1}{m-1} g'^{i \bar{i}} g'^{j \bar{j}} \phi_{i \bar{i} k} \phi_{j \bar{j} \bar{k}} - \sum_k g'^{i \bar{i}} g'^{j \bar{j}} \phi_{\bar{i} j k} \phi_{i \bar{j} \bar{k}}\\
&- \sum_k g'^{i \bar{i}} (1 + \phi_{k \bar{k}}) R_{i \bar{i} k \bar{k}} +  \sum_{i,k} R_{i \bar{i} k \bar{k}}.
\end{align*}
\indent We substitute \eqref{maxprince}, define $S:= \sum_{i,k} R_{i \bar{i} k \bar{k}}$, and obtain
\begin{align}
(m-1) f^{\frac{-1}{m-1}} \Delta f^{\frac{1}{m-1}} &\leq \alpha' m (m + \Delta \phi) - \alpha' (m + \Delta \phi) \bigg( \sum_i g'^{i \bar{i}} \bigg) \notag\\
&+ (\alpha')^2 (m + \Delta \phi) g'^{i \bar{i}} \phi_i \phi_{\bar{i}} + \alpha'' (m + \Delta \phi)g'^{i \bar{i}} \phi_i \phi_{\bar{i}} \notag\\
&+ \sum_k \frac{1}{m-1} g'^{i \bar{i}} g'^{j \bar{j}} \phi_{i \bar{i} k} \phi_{j \bar{j} \bar{k}} - \sum_k g'^{i \bar{i}} g'^{j \bar{j}} \phi_{\bar{i} j k} \phi_{i \bar{j} \bar{k}} \notag\\
&- \inf_{i,k} R_{i \bar{i} k \bar{k}} (m + \Delta \phi) \bigg( \sum_i g'^{i \bar{i}} \bigg)  + S. \label{b4_using_pos_bi}
\end{align}
\indent If $M$ has nonnegative orthogonal bisectional curvature, then $\inf_{i,k} R_{i \bar{i} k \bar{k}}$ is nonnegative and the term involving it can be dropped. We are left with
\begin{align}
(m-1) f^{\frac{-1}{m-1}} \Delta f^{\frac{1}{m-1}} &\leq  \alpha' m (m + \Delta \phi) +S - \alpha' (m + \Delta \phi) \bigg( \sum_i g'^{i \bar{i}} \bigg) \notag\\
&+ (\alpha''+ (\alpha')^2) (m+\Delta \phi) g'^{i \bar{i}} \phi_i \phi_{\bar{i}} \notag\\
&+ \sum_k \frac{1}{m-1} g'^{i \bar{i}} g'^{j \bar{j}} \phi_{i \bar{i} k} \phi_{j \bar{j} \bar{k}} - \sum_k g'^{i \bar{i}} g'^{j \bar{j}} \phi_{\bar{i} j k} \phi_{i \bar{j} \bar{k}}.
\label{pos_bisec1} \end{align}
\indent The troublesome terms are those involving third order derivatives, and we shall follow the argument of P. Guan in ~\cite{guan} to control the following quantity for a fixed $k$:
\begin{equation} \label{key_terms}
\frac{1}{m-1} g'^{i \bar{i}} g'^{j \bar{j}} \phi_{i \bar{i} k} \phi_{j \bar{j} \bar{k}} - g'^{i \bar{i}} g'^{j \bar{j}} \phi_{\bar{i} j k} \phi_{i \bar{j} \bar{k}}.
\end{equation}
\indent First, we drop mixed terms $|\phi_{\bar{i} j k}|^2$ for $i \neq j$ and obtain
$$\frac{1}{m-1} g'^{i \bar{i}} g'^{j \bar{j}} \phi_{i \bar{i} k} \phi_{j \bar{j} \bar{k}} - g'^{i \bar{i}} g'^{j \bar{j}} \phi_{\bar{i} j k} \phi_{i \bar{j} \bar{k}} \leq \frac{1}{m-1} \bigg| g'^{i \bar{i}} \phi_{i \bar{i} k} \bigg|^2 - (g'^{i \bar{i}})^2 |\phi_{i \bar{i} k}|^2.$$
\indent We recall that $\phi_{i \bar{i}}(z)$ is a locally defined real-valued function. Also, $\phi_{i \bar{i} k} = \partial_{z^k} \phi_{i \bar{i}}$ where $\partial_{z^k} = \frac{1}{2} (\partial_{x^k} - i \partial_{y^k})$. Thus
$$|\phi_{i \bar{i} k}|^2 = \frac{1}{4}(\phi_{i \bar{i} x}^2 + \phi_{i \bar{i} y}^2),$$
\noindent where we write $f_x$ for $\partial_{x^k} f$, and there is no confusion since $k$ is fixed. Thus we get
\begin{align}
 \frac{4}{m-1} g'^{i \bar{i}} g'^{j \bar{j}} \phi_{i \bar{i} k} \phi_{j \bar{j} \bar{k}} - 4g'^{i \bar{i}} g'^{j \bar{j}} \phi_{\bar{i} j k} \phi_{i \bar{j} \bar{k}} &\leq \frac{1}{m-1} \left( \sum_i g'^{i \bar{i}} \phi_{i \bar{i} x} \right)^2  - \sum_i (g'^{i \bar{i}} \phi_{i \bar{i} x})^2\notag\\
&+ \frac{1}{m-1} \left( \sum_i g'^{i \bar{i}} \phi_{i \bar{i} y} \right)^2 - \sum_i (g'^{i \bar{i}} \phi_{i \bar{i} y})^2. \label{splitting} \end{align}
\indent We shall show how to control the terms containing real derivatives in the $x$ direction. Let $I = \{ 1 \leq i \leq m : \phi_{i \bar{i} x}(p) >0 \}$ and $J = \{ 1 \leq i \leq m : \phi_{i \bar{i} x}(p) < 0 \}$. We consider the two following cases. Case 1: $I$ and $J$ are both non-empty, or case 2: either $I$ or $J$ is empty. In case $1$, we have $|I| \leq m-1$ and $|J| \leq m-1$. Using $(\sum_i^n a_i)^2 \leq n \sum a_i^2$ for $a_i \geq 0$, we can compute the following:
\begin{align*}
&\frac{1}{m-1} ( \sum_i g'^{i \bar{i}} \phi_{i \bar{i} x})^2  - \sum_i (g'^{i \bar{i}} \phi_{i \bar{i} x})^2\\
&= \frac{1}{m-1} \left( ( \sum_I g'^{i \bar{i}} \phi_{i \bar{i} x})^2 +  ( \sum_J g'^{i \bar{i}} \phi_{i \bar{i} x})^2 + 2 ( \sum_I g'^{i \bar{i}} \phi_{i \bar{i} x}) ( \sum_J g'^{i \bar{i}} \phi_{i \bar{i} x}) \right)\\
&- \sum_I (g'^{i \bar{i}} \phi_{i \bar{i} x})^2 - \sum_J (g'^{i \bar{i}} \phi_{i \bar{i} x})^2\\
&\leq \frac{1}{m-1} \left( \sum_I g'^{i \bar{i}} \phi_{i \bar{i} x} \right)^2 - \sum_I (g'^{i \bar{i}} \phi_{i \bar{i} x})^2 + \frac{1}{m-1} \left( \sum_J g'^{i \bar{i}} \phi_{i \bar{i} x} \right)^2 - \sum_J (g'^{i \bar{i}} \phi_{i \bar{i} x})^2\\
&\leq 0.
\end{align*}
\indent Case $2$ is a little bit more delicate. Without loss of generality, we assume that $J = \emptyset$. Therefore, $\phi_{i \bar{i} x}(p) > 0$ for all $i$. Using \eqref{gradientzero}, we obtain the following at $p$:
\begin{equation} \label{key_estimate}
\phi_{i \bar{i} x} \leq \sum_{j=1}^m \phi_{j \bar{j} x} = 2 Re (\frac{\partial}{\partial z^k} \Delta \phi) \leq 2 |(\Delta \phi)_k| \leq 2 \alpha' |\nabla \phi| (m + \Delta \phi).
\end{equation}
\indent We now compute
\begin{align*}
&\frac{1}{m-1} ( \sum_i g'^{i \bar{i}} \phi_{i \bar{i} x})^2  - \sum_i (g'^{i \bar{i}} \phi_{i \bar{i} x})^2\\
&= \frac{1}{m-1} ( \sum_{i=1}^{m-1} g'^{i \bar{i}} \phi_{i \bar{i} x} + g'^{m \bar{m}} \phi_{m \bar{m} x})^2  - \sum_{i=1}^m (g'^{i \bar{i}} \phi_{i \bar{i} x})^2\\
&= \frac{2}{m-1} g'^{m \bar{m}} \phi_{m \bar{m} x}  \sum_{i=1}^{m-1} g'^{i \bar{i}} \phi_{i \bar{i} x} + \frac{1}{m-1} ( \sum_{i=1}^{m-1} g'^{i \bar{i}} \phi_{i \bar{i} x})^2 - \sum_{i=1}^{m-1} (g'^{i \bar{i}} \phi_{i \bar{i} x})^2\\
&+ \frac{1}{m-1} (g'^{m \bar{m}} \phi_{m \bar{m} x})^2 - (g'^{m \bar{m}} \phi_{m \bar{m} x})^2.
\end{align*}
\indent Without loss of generality, we can assume $\phi_{m \bar{m}}(p) \geq \phi_{i \bar{i}}(p)$ for all $i$. Therefore, using \eqref{key_estimate} we have
\begin{align*}
\frac{1}{m-1} ( \sum_i g'^{i \bar{i}} \phi_{i \bar{i} x})^2  - \sum_i (g'^{i \bar{i}} \phi_{i \bar{i} x})^2 &\leq \frac{2}{m-1} g'^{m \bar{m}} \phi_{m \bar{m} x}  \sum_{i=1}^{m-1} g'^{i \bar{i}} \phi_{i \bar{i} x}\\
&\leq  \frac{8}{m-1} (\alpha')^2 |\nabla \phi|^2 (m + \Delta \phi)^2 g'^{m \bar{m}} \sum_{i=1}^{m-1} g'^{i \bar{i}}\\
&\leq \frac{8m}{m-1} (\alpha')^2 |\nabla \phi|^2 (m + \Delta \phi) \sum_{i=1}^{m} g'^{i \bar{i}} .
\end{align*}
\indent The terms involving $y$ derivatives in \eqref{splitting} can be controlled in the same way as the $x$ derivatives. Thus combining both cases and \eqref{splitting}, we obtain
\begin{align}
&\frac{1}{m-1} g'^{i \bar{i}} g'^{j \bar{j}} \phi_{i \bar{i} k} \phi_{j \bar{j} \bar{k}} - g'^{i \bar{i}} g'^{j \bar{j}} \phi_{\bar{i} j k} \phi_{i \bar{j} \bar{k}} \notag\\
&\leq \frac{4m}{m-1} (\alpha')^2 |\nabla \phi|^2 (m + \Delta \phi) \sum_{i=1}^{m} g'^{i \bar{i}}. \label{good_estimate}
\end{align}
\indent We substitute \eqref{good_estimate} into \eqref{pos_bisec1} and obtain
\begin{align}
(m-1) f^{\frac{-1}{m-1}} \Delta f^{\frac{1}{m-1}} &\leq  \alpha' m (m + \Delta \phi) +S - \alpha' (m + \Delta \phi) ( \sum_i g'^{i \bar{i}} ) \notag\\
&+ \left(\alpha''+ (\alpha')^2(1 + \frac{4m^2}{m-1}) \right) (m+\Delta \phi) |\nabla \phi|^2 ( \sum_i g'^{i \bar{i}}).
\label{pos_bisec2} \end{align}
\indent Denote $C_0 := 1 + 4m^2/(m-1)$. Following an idea of Blocki in his gradient estimate ~\cite{blocki}, we pick $\alpha(x) = (C_0)^{-1} \log x$. We know that $\alpha(\phi)$ is well-defined, since $\phi$ was renormalized such that $2 \leq \phi \leq \lambda$. This choice of $\alpha$ yields $\alpha'' + C_0 (\alpha')^2 = 0$ and hence we are left with
\begin{equation}
(m-1) f^{\frac{-1}{m-1}} \Delta f^{\frac{1}{m-1}} \leq  \frac{1}{2 C_0} m (m + \Delta \phi) +S - \frac{1}{C_0 \lambda} (m + \Delta \phi) ( \sum_i g'^{i \bar{i}} ).
\label{pos_bisec3} \end{equation}
\indent We next notice that for $B_i > 0$, the following inequality holds
$$\left( \sum_{i=1}^m \frac{1}{B_i} \right)^{m-1} \geq \frac{\sum_{i=1}^m B_i}{\prod_{i=1}^m B_i}.$$
\indent Since $g^{i \bar{i}} > 0$, we thus have at the point $p$,
\begin{equation} \label{m-1} \sum_{i} g'^{i \bar{i}} \geq  \left( \frac{m+ \Delta \phi}{f} \right)^{\frac{1}{m-1}}.\end{equation}
\indent Substituting \eqref{m-1} into \eqref{pos_bisec3} and using the definition of $A$, we get
\begin{align*}
A(m-1) &\geq \frac{1}{C_0 \lambda} (m + \Delta \phi)^{1 + \frac{1}{m+1}} - S \sup_M f^{\frac{1}{m+1}}\\
&- \left( \frac{m}{2 C_0} \sup_M f^{\frac{1}{m+1}}  \right) (m + \Delta \phi).
\end{align*}
\indent Thus there are constants $C_1$,$C_2$ under control such that
$$ (m + \Delta \phi(p))^{1 + 1/(m-1)} \leq C_1 (m + \Delta \phi(p)) + C_2.$$
\indent It follows that there exists a constant $C_3$ under control such that
$$m + \Delta \phi (p) \leq C_3.$$
\indent Now that we have control of $(m + \Delta \phi)$ at $p$, we have control of $(m + \Delta \phi)$ at all $z \in M$. Indeed,
$$ (m + \Delta \phi(z)) e^{-\alpha(\phi(z))} \leq (m + \Delta \phi(p))  e^{-\alpha(\phi(p))} \leq C_3  e^{-\alpha(\phi(p))}.$$
\indent Since $\alpha(x) = C_0^{-1} \log x$, we have
$$ (m + \Delta \phi(z)) \leq C_3 (\frac{\lambda}{2})^{1/C_0}.$$
\end{proof}

\indent By dropping the assumption on the bisectional curvature of $M$, the curvature terms break down the previous argument. It is unknown whether Proposition \ref{my_only_result} holds without this condition on the curvature. Before ending this section, we give a partial result working towards the removal of this assumption. To attempt to control these curvature terms, we strengthen our hypothesis to match the direct gradient estimate given by Blocki in ~\cite{blocki} or P. Guan in ~\cite{guan2}: we assume $f^{1/m}$ is Lipschitz continuous. In the case $m=2$, this additional assumption makes dealing with the terms \eqref{key_terms} particularly easy, and we can thus obtain the following estimate.

\begin{prop} Let $(M,g)$ be a closed, compact K\"ahler manifold of dimension $m=2$. Let $f>0$ be a positive function on $M$ such that $\inf_M \Delta f \geq -A$ for some constant $A$, and $f^{\frac{1}{2}}$ is Lipschitz. For all $\phi \in C^4(M)$ satisfying \eqref{mongeampere} such that $(\phi_{i \bar{j}} + g_{i \bar{j}})$ is positive-definite, we have
\begin{equation} | \Delta \phi | \leq C, \end{equation}
\noindent where $C$ depends on $(M,g)$, $(\sup \phi - \inf \phi)$, $A$, the Lipschitz constant of $f^{\frac{1}{2}}$, and $\sup f$. 
\end{prop}

\begin{proof}
We run the same argument as the proof of Proposition \ref{my_only_result} up until equation \eqref{b4_using_pos_bi}. In this case, we can simply let $\alpha(x)=\alpha_0 x$, where $0<\alpha_0$ is a constant. Equation \eqref{b4_using_pos_bi} becomes
\begin{align*}
f^{\frac{-1}{m-1}} \Delta ( f^{\frac{1}{m-1}} )  &\leq \alpha_0 m (m + \Delta \phi) +S - (\alpha_0 + \inf_{i,k} R_{i \bar{i} k \bar{k}}) (m + \Delta \phi) \bigg( \sum_i g'^{i \bar{i}} \bigg)\\ 
&+ (\alpha_0)^2 (m + \Delta \phi)  g'^{i \bar{i}} \phi_i \phi_{\bar{i}} + \sum_k \frac{1}{m-1} g'^{i \bar{i}} g'^{j \bar{j}} \phi_{i \bar{i} k} \phi_{j \bar{j} \bar{k}}\\
&- \sum_k g'^{i \bar{i}} g'^{j \bar{j}} \phi_{\bar{i} j k} \phi_{i \bar{j} \bar{k}}.
\end{align*}
\indent We see that if we choose $\alpha_0 > \inf_{i,k} R_{i \bar{i} k \bar{k}}$, the coefficient on the third term is negative. To eliminate the $\alpha_0^2$ term, we substitute the gradient equation \eqref{gradientzero}:
\begin{align*}
f^{\frac{-1}{m-1}} \Delta ( f^{\frac{1}{m-1}} )  &\leq \alpha_0 m (m + \Delta \phi) +S - (\alpha_0 + \inf_{i,k} R_{i \bar{i} k \bar{k}}) (m + \Delta \phi) \bigg( \sum_i g'^{i \bar{i}} \bigg)\\ 
&+ (m + \Delta \phi)^{-1} g'^{i \bar{i}} (\Delta \phi)_i (\Delta \phi)_{\bar{i}} + \sum_k \frac{1}{m-1} g'^{i \bar{i}} g'^{j \bar{j}} \phi_{i \bar{i} k} \phi_{j \bar{j} \bar{k}}\\
&- \sum_k g'^{i \bar{i}} g'^{j \bar{j}} \phi_{\bar{i} j k} \phi_{i \bar{j} \bar{k}}.
\end{align*}
\indent Using Cauchy-Bunyakowsky-Schwarz, we obtain
\begin{align*}
(m + \Delta \phi)^{-1} g'^{i \bar{i}} (\Delta \phi)_i (\Delta \phi)_{\bar{i}} &= (m + \Delta \phi)^{-1} \sum_i g'^{i \bar{i}} (\sum_k \phi_{k \bar{k} i})(\sum_k \phi_{k \bar{k} \bar{i}}) \\
&\leq (m + \Delta \phi)^{-1} \sum_i g'^{i \bar{i}} \left( \sum_k \frac{|\phi_{k \bar{k} i}|^2}{(1 + \phi_{k \bar{k}})} \right) \left(  \sum_k (1 + \phi_{k \bar{k}}) \right)\\
&= \sum_{i,k} g'^{i \bar{i}} g'^{k \bar{k}} | \phi_{k \bar{k} i} |^2\\
&\leq \sum_{i,j,k} g'^{i \bar{i}} g'^{j \bar{j}} | \phi_{i \bar{j} k} |^2.
\end{align*}
\indent We are left with
\begin{align*}
f^{\frac{-1}{m-1}} \Delta ( f^{\frac{1}{m-1}} )  &\leq \alpha_0 m (m + \Delta \phi) +S - (\alpha_0 + \inf_{i,k} R_{i \bar{i} k \bar{k}}) (m + \Delta \phi) \bigg( \sum_i g'^{i \bar{i}} \bigg)\\ 
&+ \sum_k \frac{1}{m-1} g'^{i \bar{i}} g'^{j \bar{j}} \phi_{i \bar{i} k} \phi_{j \bar{j} \bar{k}}.
\end{align*}
\indent It is at this point that we use the hypotheses that $m=2$ and $f^{1/2}$ is Lipschitz continuous. From taking the derivative of both sides of $(\det g'_{i \bar{j}})^{1/m} = (f \det g_{i \bar{j}})^{1/m}$, we see that $g'^{i \bar{j}} \phi_{i \bar{j} k} = 2 f^{-1/2} \partial_k f^{1/2}$. Therefore,
$$f^{-1} \Delta f  \leq 2 \alpha_0 (2 + \Delta \phi) +S - (\alpha_0 +\inf_{i,k} R_{i \bar{i} k \bar{k}} ) (2 + \Delta \phi) \bigg( \sum_i g'^{i \bar{i}} \bigg) + 4 \frac{|\nabla f^{1/2}|^2}{f}.$$
\indent Since we choose $\alpha_0$ such that $\alpha_0 + \inf_{i,k} R_{i \bar{i} k \bar{k}} >0$, we use \eqref{m-1} and get
$$-A \leq S \sup_M f + 4 ||\nabla f^{1/2}||_\infty^2 + 2 \alpha_0 (\sup_M f) (2 + \Delta \phi) - (\alpha_0 + \inf_{i,k} R_{i \bar{i} k \bar{k}}) (2 + \Delta \phi)^2.$$
\indent As shown in the previous argument, it follows that $(m + \Delta \phi) \leq C$.

\end{proof}

\section{Dirichlet Problem in $\mathbb{C}^m$}

\indent As an application of the {\it a priori} estimates shown previously, we shall solve a Dirichlet problem in $\mathbb{C}^m$, following the footsteps of ~\cite{ckns}. In order to prove Theorem \ref{dirichlet}, we will make use of  estimates previously established in the literature and combine them with our result.
\newline
\indent Let $\Omega$ be a strongly pseudo-convex domain and $u$ be a strictly pluri-subharmonic solution to \eqref{thm3}, where $f>0$. We let $\psi$ be a strictly pluri-subharmonic function on $\Omega$ such that $\psi = 0$ on $\partial \Omega$ and
$$ \det (\psi_{i \bar{j}}) > \sup_{\bar{\Omega}} f \geq \det ( u_{i \bar{j}}).$$
\indent By a maximum principle such as the one given in ~\cite{ckns}, we have have $\psi \leq u$ in $\overline{\Omega}$. To get a upper bound, we solve the Laplace equation for a harmonic function $h$: $\Delta h = 0$ in $\Omega$ and $h=0$ on $\partial \Omega$. Then since $\Delta u \geq 0$ in $\Omega$, we have $u \leq h$. Using $\psi \leq u \leq h$ in $\Omega$ and $\psi=u=h$ on $\partial \Omega$, we can obtain $|\nabla u(z)| \leq \max \{ |\nabla \psi(z)|, |\nabla h(z)| \}$ for all $z \in \partial \Omega$. Since $\psi$ and $h$ depend only on $\Omega$, we have a gradient estimate on the boundary. 
\newline
\indent To push the interior gradient estimate to the boundary as done in ~\cite{ckns}, we need to introduce the additional hypothesis that $f^{1/m}$ is Lipschitz. It is an open question to determine whether this assumption can be improved to requiring that $f^{1/(m-1)}$ is Lipschitz. Such a result would lead to a more natural statement for Theorem \ref{dirichlet}, which would be analogous to the result given in ~\cite{guan} for the real Monge-Amp\`{e}re equation. In our case, we assume the hypothesis that $|\nabla f^{1/m}| \leq A_1$, and thus have
$$|| u ||_{C^1(\overline{\Omega})} \leq C.$$
\indent We can obtain second order estimates of $u$ from our current result. Define $\phi := u - |z|^2$. We have $\det( \delta_{i \bar{j}} + \phi_{i \bar{j}} )= f \det \delta_{i \bar{j}}$. Consider the test function $H = (m+ \Delta \phi) e^{- \alpha (\phi)}$ from Proposition \ref{my_only_result}. If $H$ attains its maximum on $\partial \Omega$, then $|m+ \Delta \phi| \leq C || m+ \Delta \phi||_{L^\infty(\partial \Omega)}$. If $H$ attains its maximum at $p \in \Omega$, then we obtain $|m+ \Delta \phi| \leq C$ if we follow the proof of Proposition \ref{my_only_result} with $g_{i \bar{j}} = \delta_{i \bar{j}}$. Therefore, we have
$$||\Delta u||_{L^\infty(\overline{\Omega})} \leq C (1 + || \Delta u||_{L^\infty(\partial \Omega)}).$$
\indent Control of second order derivatives on $\partial \Omega$ follows from the argument in ~\cite{ckns}, and this argument relies on a $C^1$ estimate and $f^{1/m}$ to be Lipschitz. Therefore, we have the following result.

\begin{prop} \label{apriori_dirichlet}
Let $\Omega$ be a strongly pseudo-convex domain in $\mathbb{C}^m$. Let $f: \Omega \rightarrow \mathbb{R}$ be a function such that $f > 0$, $|\nabla f^{1/m}| \leq A_1$, and $\Delta f^{\frac{1}{m-1}} \geq -A_2$. Suppose there exists a strictly pluri-subharmonic solution $u \in C^\infty(\bar{\Omega})$ such that 
$$ \det u_{i \bar{j}}(z) = f(z) \ \ {\rm in } \ \Omega,$$
$$ u = 0 \ \ {\rm on} \ \partial \Omega.$$
\indent Then there exists a constant $C$ which depends only on $\Omega$, $\sup(f)$, $A_1$ and $A_2$ such that
$$ ||u||_{C^1(\overline{\Omega})} + ||\Delta u||_{L^\infty(\overline{\Omega})} \leq C.$$
\
\end{prop}
\ 
\indent Using Proposition \ref{apriori_dirichlet}, we shall now prove Theorem \ref{dirichlet}. The strategy will be to solve the non-degenerate Dirichlet problem for $f>0$, and then use a limiting process. Let $g_\epsilon = f^{\frac{1}{m-1}} + \epsilon$, with $\epsilon >0$. We extend $f$ such that it is defined on all of $\mathbb{C}^m$. Let $\gamma_\rho = \gamma(|z|/\rho)$, where $\gamma: \mathbb{C}^m \rightarrow \mathbb{R}$ is a $C^\infty$ function of compact support such that $0 \leq \gamma \leq 1$, $\int_{\mathbb{C}^m} \gamma = 1$. We define $h_{\epsilon,\rho}: \Omega \rightarrow \mathbb{R}$ in the following way: 
$$h_{\epsilon,\rho}(x) = ( g_\epsilon * \gamma_\rho (x))^{m-1}.$$
\indent Since $\overline{\Omega}$ is compact, we know that $g_\epsilon * \gamma_\rho \rightarrow g_\epsilon$ uniformly on $\overline{\Omega}$. For $\rho, \epsilon >0$ small enough, it can be shown that 
$$ |\nabla h_{\epsilon,\rho}^{1/m}| \leq 2 A_1 , \ \ \Delta h_{\epsilon,\rho}^{1/(m-1)} \geq -A_2.$$
\indent Now, we consider the non-degenerate Monge-Amp\`{e}re Dirichlet problem
$$\det (u_{\epsilon,\rho})_{i \bar{j}} = h_{\epsilon,\rho} \ \ {\rm in } \ \Omega,$$
$$ u_{\epsilon,\rho} = 0 \ \ {\rm on} \ \partial \Omega.$$
\indent By ~\cite{ckns}, since $h_{\epsilon,\rho}$ is smooth, we know that there exists a smooth strictly pluri-subharmonic solution $u_{\epsilon,\rho}$. By Proposition \ref{apriori_dirichlet}, we have
$$||u_{\epsilon,\rho}||_{C^1{(\overline{\Omega})}} + ||\Delta u_{\epsilon,\rho}||_\infty \leq C,$$
\noindent for some constant $C$ independent of $\epsilon$ and $\rho$. We let $\rho \rightarrow 0$ and obtain a strictly pluri-subharmonic solution $u_{\epsilon}$ of 
$$\det (u_{\epsilon})_{i \bar{j}} = (f^{\frac{1}{m-1}} + \epsilon)^{m-1} \ \ {\rm in } \ \Omega,$$
$$ u_{\epsilon} = 0 \ \ {\rm on} \ \partial \Omega.$$
\noindent such that $||u_{\epsilon}||_{C^1{(\overline{\Omega})}} + ||\Delta u_{\epsilon}||_\infty \leq C.$ Finally, we let $\epsilon \rightarrow 0$ and obtain a pluri-subharmonic solution $u$ of \eqref{dirichlet} such that $||u||_{C^1{(\overline{\Omega})}} + ||\Delta u||_\infty \leq C.$ Uniqueness is well-known and can be found in ~\cite{ckns}.

\end{document}